\documentclass[11pt]{article}

\usepackage[utf8]{inputenc}
\usepackage{geometry}
\usepackage{graphicx}

\makeatletter
\@ifundefined{showcaptionsetup}{}{%
 \PassOptionsToPackage{caption=false}{subfig}}
\usepackage{subfig}
\makeatother
\usepackage{booktabs}
\usepackage{array}
\usepackage{paralist}
\usepackage{verbatim} 
\usepackage{subfig}
\usepackage[usenames,svgnames,dvipsnames]{xcolor}
\usepackage{amsmath}
\usepackage{amssymb}
\usepackage{amsthm}
\usepackage{mathrsfs}
\usepackage{amsbsy}
\usepackage{multirow}
\usepackage{tikz}

\usepackage{babel}
\usepackage[pdftex,plainpages=false,colorlinks=true,citecolor=DarkGreen,linkcolor=Blue,urlcolor=black,filecolor=black,bookmarksopen=true]{hyperref}
\usepackage{setspace}
\usepackage{enumitem}
\hypersetup{pdftitle={Convergence of the CEM-GMsFEM for compressible flow in highly heterogeneous media},
 pdfauthor={Leonardo A. Poveda and Shubin Fu and Eric T. Chung and Lina Zhao}}
 
\numberwithin{equation}{section}

\usepackage{fancyhdr}
\usepackage[titles,subfigure]{tocloft} 

\newtheorem{thm}{Theorem}[section]

\newtheorem{lem}[thm]{Lemma}

\theoremstyle{definition}
\newtheorem{defn}[thm]{Definition}
\newtheorem{example}[thm]{Example}

\theoremstyle{remark}

\numberwithin{equation}{section}

\def\ie{\emph{i.e.}}

\def\fall{\mbox{for all }}
\def\feac{\mbox{for each }}
\def\spn{\mathrm{span}}
\def\argmin{\mathrm{argmin}}

\def\fract{\mathrm{frac}}

\def\Or{{\cal O}}

\def\ms{\mathrm{ms}}

\def\Om{\Omega}
\def\OmT{\Om_{T}}
\def\pOm{\p\Om}
\def\GD{\Gamma_{D}}
\def\GN{\Gamma_{N}}
\def\oi{\omega_{i}}

\def\Ki{K_{i}}
\def\Kim{K_{i,m}}

\def\lij{\lambda^{(i)}_{j}}

\def\x{\mathrm{x}}
\def\p{\partial}
\def\pt{\p_{t}}
\def\dive{\nabla\cdot}
\def\gd{\nabla}
\def\k{\kappa}
\def\wk{\widetilde{\k}}
\def\fkm{\frac{\k}{\mu}}
\def\tfkm{\tfrac{\k}{\mu}}

\def\R{\mathbb{R}}
\def\Rd{\R^{d}}
\def\L{\mathrm{L}}
\def\H{\mathrm{H}}
\def\V{\mathrm{V}}
\def\Lt{\L^{2}}
\def\Ho{\H^{1}}
\def\Hoz{\H^{1}_{0}}
\def\Vz{\V_{0}}

\def\nV{\widetilde{\V}}

\def\Vaux{\V_{\mathrm{aux}}}
\def\Viaux{\V_{\mathrm{aux}}^{(i)}}
\def\Vms{\V^{\ms}}

\def\Vh{\V^{h}}

\def\Th{{\cal T}^{h}}
\def\TH{{\cal T}^{H}}
\def\coar{\tiny{\rm{coarse}}}
\def\fine{\tiny{\rm{fine}}}
\def\tme{\tiny{\rm{time}}}
\def\Nc{N_{\coar}} 
\def\Nf{N_{\fine}} 
\def\Nt{N_{\tme}} 
\def\dtn{\Delta_{n}}
\def\pn{p_{n}}
\def\pnm{p_{n-1}}
\def\ph{p^{h}}
\def\pp{\hat{p}}
\def\phn{\ph_{n}}
\def\pnk{p^{k}_{n}}
\def\pnko{p^{k+1}_{n}}

\def\phnm{\ph_{n-1}}
\def\phnk{p^{h,k}_{n}}
\def\pink{p^{i,k}_{n}}
\def\pinm{p^{i}_{n-1}}
\def\pms{p^{\ms}}
\def\pmsn{\pms_{n}}
\def\pmsnm{\pms_{n-1}}

\def\pw{\hat{w}}

\def\rref{\rho_{\mathrm{ref}}}
\def\pref{p_{\mathrm{ref}}}
\def\vij{\varphi^{(i)}_{j}}

\def\psijms{\psi_{j,\ms}^{(i)}}

\def\wrni{\widetilde{r}_{n}^{(i)}}
\def\rni{r_{n}^{(i)}}
\def\Rms{r^{\ms}}
\def\Rmsn{\Rms_{n}}
\def\Fjnk{F_{j}^{n,k}}
\def\Fnk{{\mathbf F}^{n,k}}
\def\Jjink{J_{ji}^{n,k}}
\def\Jnk{{\mathbf J}^{n,k}}
\def\dpnk{\boldsymbol{\delta}_{p^{n,k}}}

\providecommand{\keywords}[1]{\small\textbf{Keywords:} #1}

\title{Convergence of the CEM-GMsFEM for compressible flow in highly heterogeneous media}
\author{Leonardo A. Poveda\thanks{Department of Mathematics, The Chinese University of Hong Kong, Shatin, Hong Kong SAR, China} \and Shubin Fu\thanks{Eastern Institute for Advanced Study, Ningbo, China} \and Eric T. Chung\thanks{Department of Mathematics, The Chinese University of Hong Kong, Shatin, Hong Kong SAR, China} \and Lina Zhao\thanks{Department of Mathematics, City University of Hong Kong, Hong Kong SAR, China}}

\begin{document}
\maketitle
\begin{abstract}
This paper presents and analyses a Constraint Energy Minimization Generalized Multiscale Finite Element Method (CEM-GMsFEM) for solving single-phase non-linear compressible flows in highly heterogeneous media. 
The construction of CEM-GMsFEM hinges on two crucial steps: First, the auxiliary space is constructed by solving local spectral problems, where the basis functions corresponding to small eigenvalues are captured. Then 
the basis functions are obtained by solving local energy minimization problems over the oversampling domains using the auxiliary space.   The basis functions have exponential decay outside the corresponding local oversampling regions.  The convergence of the proposed method is provided, and we show that this convergence only depends on the coarse grid size and is independent of the heterogeneities. An online enrichment guided by \emph{a posteriori} error estimator is developed to enhance computational efficiency. Several numerical experiments on a three-dimensional case to confirm the theoretical findings are presented, illustrating the performance of the method and giving efficient and accurate numerical
\end{abstract}


\section{Introduction}
\label{sec:introduction}
The numerical solution of non-linear partial differential equations defined on domains with multiscale and heterogeneous properties is an active research subject in the scientific community. The subject is related to several engineering applications, such as composite materials, porous media flow, and fluid mechanics. A common feature for all these applications is that they are very computationally challenging and often impossible to solve within an acceptable tolerance using standard fine-scale approximations due to the disparity between scales that need to be represented and the inherent nonlinearities. For this reason, coarse-grid computational models are often used. These approaches are usually referred to as multiscale methods in the literature, among which we may mention: Multiscale Finite Element Method  \cite{hou1997multiscale}, the Variational Multiscale Method \cite{hughes1998variational}, Mixed Multiscale Finite Element Method \cite{chen2003mixed}, Mixed Mortar Multiscale Finite Element Method \cite{Arbogast07}, the two-scale Finite Element Method \cite{matache2000homogenization}, and the Multiscale Finite Volume method \cite{jenny2003multi,jenny2004adaptive}. The aforementioned methods share model reduction techniques, using different structures to find multiscale solutions, especially in many practical applications such as fluid flow simulations, for instance, \cite{alotaibi2022generalized,chen2020generalized,park2005mixed,kim2007multiscale,hajibeygi2009multiscale,wang2014algebraic,fu2022generalized,vasilyeva2019nonlocal}. In particular, we consider a family of an extended version of MsFEM, known as the Generalized Multiscale Finite Element Method (GMsFEM) that was first introduced by \cite{efendiev2013generalized,efendiev2014generalized,chung2014generalized,chung2014adaptive}. The main idea of GMsFEM is to construct localized basis functions by solving local spectral problems that are used to approximate the solution on a coarse grid incorporating fine-scale features. Following this, we construct an auxiliary space associated with local spectral problems in the coarse grid. The first few eigenfunctions corresponding to small eigenvalues (the convergence depends on the decay of the spectral problems \cite{efendiev2013generalized}) are considered as the multiscale basis functions.  In this paper, we extend the Constraint Energy Minimization Generalized Multiscale Finite Element Method (CEM-GMsFEM) developed in \cite{chung2018constraint,li2019constraint} to solve single-phase nonlinear compressible flow. The key ideas of the method can be summarized as follows: First, we construct the auxiliary basis functions by solving local spectral problems. Then, by using oversampling techniques and localization (cf. \cite{malqvist2014localization}), we solve an appropriate energy subject to some constrainted oversampling regions to find the required basis functions. Finally, the resulting basis functions are shown to have exponential decay away from the target coarse element, and therefore, they are localizable. 
Then we rigorously analyze the convergence error estimates for the proposed scheme. 
Our theories indicate that  the convergence rate behaves as $H/\Lambda$, where $H$ denotes the coarse-grid size and $\Lambda$ is proportional to the minimum (taken over all coarse regions) of the eigenvalues that the corresponding eigenvectors are not included in the coarse space. Since the problems under consideration are nonlinear, some novel methodologies shall be incorporated to overcome the difficulties present in the analysis. Several numerical experiments are carried out to demonstrate the capabilities and efficiency of the proposed method.

The outline of the paper is following. In Section \ref{sec:mathematical}, we briefly derive a mathematical model for compressible fluid flows in porous media. Section \ref{sec:construction} is devoted to constructing the offline multiscale space and framework of CEM-GMsFEM. Section \ref{sec:convergence} presents our  convergence analysis for the proposed method. Numerical experiments are presented in Section \ref{sec:numerical}. Finally, concluding remarks and future perspectives are given in Section \ref{sec:conclusion}.

\section{Mathematical model for compressible fluid flow}
\label{sec:mathematical}
In this section, we consider the governing equations of the single-phase, non-linear compressible fluid flow processes in a porous medium that are defined by
\begin{equation}
\label{eq:main-prob}
\left\{\begin{split}
\pt(\phi\rho)-\dive\left(\fkm\rho\gd p \right)&=q,\quad\mbox{in }\OmT:=\Om\times(0,T],\quad T>0,\\
\fkm\rho\gd p\cdot n &= 0,\quad\mbox{on }\GN\times(0,T],\\
p & = p^{D},\quad\mbox{on }\GD\times(0,T],\\
p & = p_{0},\quad\mbox{on }\Om\times\{t=0\}.
\end{split}\right.
\end{equation}
For simplicity of presentation, let $\Om\subset \Rd$ be the computational domain with a boundary defined by $\pOm=\GD\cup\GN$. We will henceforth neglect the gravity effects and capillary forces and assume that $\phi$, the porosity of the medium, is assumed to be a constant. We aim to seek the fluid pressure $p$, $\k$ denotes the permeability field that may be highly heterogeneous, such that $\k_{0}\leq\k\leq\k_{1}$, where $0<\k_{0}<\k_{1}<\infty$; and $\mu$ is the constant fluid viscosity. The fluid density $\rho$ is a function of the fluid pressure $p$ defined as
\begin{equation}
\rho(p)=\rref e^{c(p-\pref)},
\end{equation}
where $\rref$ is the given reference density and $\pref$ is the reference pressure. Finally, $n$ denotes the outward unit-normal vector on $\pOm$. 

For a sub-domain $D\subset\Om$, let  $\V:=\Ho(D)$ be the standard Sobolev spaces endowed with the norm $\|\cdot\|_{1,D}$. We further denote by $(\cdot,\cdot)$ and $\|\cdot\|_{0,D}$ the inner product and norm, respectively in $\L^{2}(D)$. The subscript $D$ will be omitted whenever $D=\Om$. In addition, we use the space $\Vz:=\Hoz(D)$, which is a subspace of $\V$ made of functions that vanish at $\GD$. Finally, let $\L^{2}(0,T;\L^{2}(D))$ denote the set of functions with norm
\[
\|v\|_{\L^{2}(0,T;\L^{2}(D))}=\left(\int_{0}^{T}\|v(\cdot,t)\|^{2}_{0,D}dt\right)^{1/2}.
\]
Throughout the paper, $a\preceq b$ means there exists a positive constant $C$ independent of the mesh size such that $a\leq Cb$.

\subsection{A finite element approximation}
In this subsection, we introduce the notions of fine and coarse grids to discretize the problem \eqref{eq:main-prob}. Let $\TH$ be a usual conforming partition of the computational domain $\Om$ into coarse block $K\in\TH$ with diameter $H$. Then, we denote this partition as the coarse grid and assume that each coarse element is partitioned into a connected union of fine-grid blocks. In this case, the fine grid partition will be denoted by $\Th$, and is, by definition, a refinement of the coarse grid $\TH$, such that $h\ll H$. We shall denote $\{x_{i}\}_{i=1}^{\Nc}$ as the vertices of the coarse grid $\TH$, where $\Nc$ denotes the number of coarse nodes. We define the neighborhood of the node $x_{i}$ by 
\[
\oi=\bigcup\{K_{j}\in\TH:x_{i}\in\overline{K}_{j}\}.
\]
In addition, for CEM-GMsFEM considered in this paper, we have that given a coarse block $\Ki$, we represent the oversampling region $\Kim\subset\Om$ obtained by enlarging $\Ki$ with $m\geq1$ coarse grid layers, see Fig.~\ref{fig:grid}.
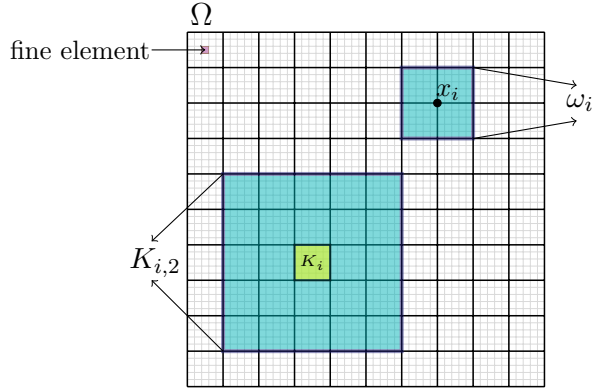
\begin{figure}[h!]
\centering
\begin{tikzpicture}[scale=4.7]
\draw[step=0.02cm,color=Gray, line width = 0.01mm,opacity=0.3] (0,0) grid (1,1);
\draw[step=0.1cm,color=Black,line width = 0.2mm] (0,0) grid (1,1);
\draw [draw=MidnightBlue, line width=1.5pt, fill=BlueGreen, opacity=0.50] (0.1,0.1) rectangle (0.6,0.6);
\draw [draw=Sepia, line width=0.02mm, fill=Mulberry, opacity=0.50] (0.04,0.94) rectangle (0.06,0.96);
\draw [draw=MidnightBlue, line width=1.5pt, fill=BlueGreen, opacity=0.50] (0.6,0.7) rectangle (0.8,0.9);
\node at (-0.09,0.35) {$K_{i,2}$};
\draw [draw=black, line width=0.3mm, fill=Yellow, opacity=0.5] (0.3,0.3) rectangle (0.4,0.4);
\node at (0.35,0.35) {\tiny{$\Ki$}};
\node at (-0.3,0.95) {\small{\mbox{fine element}}};
\filldraw (0.7,0.8) circle (0.3pt);
\node at (0.73,0.83) {\small{$x_{i}$}};
\node at (1.10,0.8) {$\oi$};
\node at (0.04,1.05) {\large{$\Om$}};
\coordinate (A) at (-0.1,0.3);
\coordinate (B) at (-0.1,0.41);
\coordinate (C) at (0.1,0.1);
\coordinate (D) at (0.1,0.6);
\coordinate (E) at (0.05,0.95);
\coordinate (F) at (-0.1,0.95);

\coordinate (G) at (1.09,0.75);
\coordinate (H) at (1.09,0.85);
\coordinate (I) at (0.8,0.7);
\coordinate (J) at (0.8,0.9);

\draw [<-] (A) -- (C);
\draw [<-] (B) -- (D);
\draw [<-] (E) -- (F);
\draw [<-] (G) -- (I);
\draw [<-] (H) -- (J);
\end{tikzpicture}
\caption{Illustration of the $2$D multiscale grid with a typical coarse element $\Ki$ and oversampling domain $K_{i,2}$, the fine grid element and neighborhood $\oi$ of the node $x_{i}$.}\label{fig:grid}
\end{figure}
We consider the linear finite element space $\Vh$ associated with the grid $\Th$, where the basis functions in this space are the standard Lagrange basis functions defined as $\{\eta^{i}\}_{i=1}^{\Nf}$. Then, the semi-discrete finite element approximation to \eqref{eq:main-prob} on the fine grid is to find $\ph\in\Vh$ such that
\begin{equation}
\label{eq:weak-prob}
(\phi\pt\rho(\ph),v)+\left(\tfkm\rho(\ph)\gd\ph,\gd v\right)=(q,v),\quad\feac v\in\Vh.
\end{equation}
We can now define the fully-discrete scheme for the discrete formulation \eqref{eq:weak-prob}. Let $0=t_{0}<t_{1}<\cdots<t_{\Nt-1}<t_{\Nt}=T$ be a partition of the interval $[0,T]$, with time-step size given by $\dtn=t_{n}-t_{n-1}$,  for $n=1,\dots,\Nt$, where $\Nt$ is an integer. The backward Euler time integration leads to finding $\phn$ such that
\begin{equation}
\label{eq:fweak-prob}
(\phi\rho(\phn),v)-(\phi\rho(\phnm),v)+\dtn\left(\tfkm\rho(\phn)\gd\phn,\gd v\right)=\dtn(q,v),\quad\feac v\in\Vh.
\end{equation}
Linearization of \eqref{eq:fweak-prob} via Newton-Raphson iteration yields a iterative linear matrix problem, 
\[
\Jnk\dpnk=-\Fnk,
\]
where $\Jnk:=[\Jjink]_{i,j=1}^{\Nf}$ denotes the Jacobi matrix, with entries 
\[
\Jjink:= (\phi\rho(\phnk)\eta_{i},\eta_{j})+\dtn\left(\fkm\rho(\phnk)\gd\eta_{i},\gd\eta_{j}\right)+\dtn\left(c\fkm\eta_{i}\rho(\phnk)\sum_{i}\pink\gd\eta_{i},\gd\eta_{j}\right),
\]
$\Fnk:=[\Fjnk]_{j=1}^{\Nf}$ is the residual with entries 
\begin{align*}
\Fjnk&=\left(\phi\rho\left(\sum_{i=1}^{\Nf}\pink\eta_{i}\right),\eta_{j}\right)-\left(\phi\rho\left(\sum_{i=1}^{\Nf}\pinm\eta_{i}\right),\eta_{j}\right)
+\dtn\left(\fkm\rho\left(\sum_{i=1}^{\Nf}\pink\eta_{i}\right)\sum_{i=1}^{\Nf}\pink\gd\eta_{i},\gd\eta_{j}\right)\\
&\quad-\dtn(q,\eta_{j})=0,
\end{align*}
and $\pnko=\pnk+\dpnk$, where $\phnk=\sum_{i}\pink\eta_{i}$ and $\phnm=\sum_{i}\pinm\eta_{i}$, with $k$ and $k-1$ the new and old Newton iteration. Here $\{\eta^{i}\}_{i=1}^{\Nf}$ represents the finite element basis functions for $\Vh$.

\section{Construction of CEM-GMsFEM basis functions}
\label{sec:construction}

This section will describe the construction of CEM-GMsFEM basis functions using the framework of \cite{chung2018constraint} and \cite{li2019constraint}. This procedure can be divided into two stages. The first stage involves constructing the auxiliary spaces by solving a local spectral problem in each coarse element $K$, see \cite{efendiev2013generalized}. The second stage is to provide the multiscale basis functions by solving some local constraint energy minimization problems in oversampling regions. 

\subsection{Auxiliary basis function}
In this subsection, we present the construction of the auxiliary multiscale basis functions by solving the local eigenvalue problem for each coarse element $K_{i}$. We consider $\V(\Ki):=\V\scriptstyle\big|_{\Ki}$ the restriction of the space $\V$ to the coarse element $\Ki$. We solve the following local eigenvalue problem: find $\{\lij,\vij\}$ such that 
\begin{equation}
\label{eq:lspec-prob}
a_{i}(\vij,w)=\lij s_{i}(\vij,w),\quad\feac w\in\V(K_{i}),
\end{equation}
where 
\[
a_{i}(v,w):=\int_{\Ki}\k\rho(p_{0})\gd v\cdot\gd wd\x, \quad s_{i}(v,w):=\int_{K_{i}}\widetilde{\k}vwd\x.
\]
Here $\wk=\rho(p_{0})\k\sum_{i=1}^{\Nc}|\gd\chi_{i}|^{2}$, where $\Nc$ is the total number of neighborhoods, $p_{0}$ is the pressure $p$ at the initial time and $\{\chi_{i}\}$ is a set of partition of unity functions for the coarse grid $\TH$, see \cite{babuska1997partition}. The problem  defined above is solved on the fine grid in the actual computation. We assume that the eigenfunctions satisfy the normalized condition $s_{i}(\vij,\vij)=1$. We let $\lij$ be the eigenvalues of \eqref{eq:lspec-prob} arranged in ascending order. We shall use the first $L_{i}$ eigenfunctions to construct the local auxiliary multiscale space $\Viaux:=\{\vij:1\leq j\leq L_{i}\}$. We can define the global auxiliary multiscale space as  $\Vaux:=\bigoplus_{i=1}^{\Nc}\Viaux$.

For the local auxiliary space $\Viaux$, the bilinear form $s_{i}$ given above defines an inner product with norm $\|v\|_{s(\Ki)}=s_{i}(v,v)^{1/2}$. Then, we can define the inner product and norm for the global auxiliary multiscale space $\Vaux$, which are defined by
\[
s(v,w)=\sum_{i=1}^{\Nc}s_{i}(v,w),\quad \|v\|_{s}:=s(v,v)^{1/2},\quad \feac v,w\in\Vaux. 
\]
To construct the CEM-GMsFEM basis functions, we use the following definition.
\begin{defn}[\cite{chung2018constraint}]
Given a function $\vij\in\Vaux$, if a function $\psi\in\V$ satisfies 
\[
s(\psi,\vij):=1,\quad s(\psi,\varphi_{j'}^{(i')})=0,\quad\mbox{if }j'\neq j\mbox{ or }i'\neq i,
\]
then, we say that is $\vij$-orthogonal where $s(v,w)=\sum_{i=1}^{N}s_{i}(v,w)$.
\end{defn}
Now, we define $\pi:\V\to\Vaux$ to be the projection with respect to the inner product $s(v,w)$. So, $\pi$ is defined by
\[
\pi(v):=\sum_{i=1}^{\Nc}\pi_{i}(v)=\sum_{i=1}^{\Nc}\sum_{j=1}^{L_{i}}s_{i}(v,\vij)\vij,\quad\feac v\in\V,
\]
where $\pi_{i}:L^{2}(\Ki)\to\Viaux$ denotes the projection with respect to inner product $s_{i}(\cdot,\cdot)$. The null space of the operator $\pi$ is defined by $\nV=\{v\in\V:\pi(v)=0\}$.
Now, we will construct the multiscale basis functions. Given a coarse block $\Ki$, we denote the oversampling region $\Kim\subset\Om$ obtained by enlarging $\Ki$ with an arbitrary number of coarse grid layers $m\geq1$, see Figure \ref{fig:grid}. Let $\Vz(\Kim):=\Hoz(\Kim)$. Then, we define the multiscale basis function \begin{equation}
\label{eq:main-argmin}
\psijms=\argmin\{a(\psi,\psi):\psi\in\Vz(\Kim),\,\psi\mbox{ is $\vij$-orthogonal}\},
\end{equation}
where $\V(\Kim)$ is the restriction of $\Vz$ in $\Kim$ and $\Vz(\Kim)$ is the subspace of $\V(\Kim)$ with zero trace on $\p\Kim$. The multiscale finite element space $\Vms$ is defined by
\[
\Vms=\spn\{\psijms:1\leq j\leq L_{i}, 1\leq i\leq \Nc\}.
\]
By introducing the Lagrange multiplier, the problem \eqref{eq:main-argmin} is equivalent to the explicit form: find $\psijms\in\Vz(\Kim)$, $\lambda\in\Viaux(\Ki)$ such that
\[
\begin{cases}
a(\psijms,\eta)+s(\eta,\lambda)&=0,\quad\fall \eta\in \V(\Kim),\\
s(\psijms-\vij,\nu)&=0,\quad\fall \nu\in\Viaux(\Kim),
\end{cases}
\]
where $\Viaux(\Kim)$ is the union of all local auxiliary spaces for $\Ki\subset\Kim$. Thus, the semi-discrete multiscale approximation reads as follows: find $\pmsn\in\Vms$ such that
\begin{equation}
\label{eq:sms-prob}
(\phi\pt\rho(\pmsn),v)+\left(\tfkm\rho(\pmsn)\gd\pmsn,\gd v\right)=(q,v),\quad\feac v\in\Vms.
\end{equation}
Using the backward Euler time-stepping scheme, we have a full-discrete formulation: find $\pmsn\in\Vms$ such that
\begin{equation}
\label{eq:fms-prob}
(\phi\rho(\pmsn),v)-(\phi\rho(\pmsnm),v)+\dtn\left(\tfkm\rho(\pmsn)\gd\pmsn,\gd v\right)=\dtn(q,v),\quad\feac v\in\Vms.
\end{equation}

\section{Convergence analysis}
\label{sec:convergence}
In this section, we establish the estimates of the convergence order of the proposed method. 

\subsection{Error estimates}
In this subsection, we will present the convergence error estimates for the semi-discrete scheme \eqref{eq:sms-prob}. The analysis consists of two main steps. First, we derive the error estimate for the difference between the exact solution and its corresponding elliptic projection. Second, we estimate the difference between the solution of \eqref{eq:main-prob} and solution of \eqref{eq:sms-prob} by the difference between exact solution and the elliptic projection solution of problem \eqref{eq:main-prob}.

To begin, we let $\pp\in\Vms$ be the elliptic projection of the function $p\in\V$ that is defined by
\begin{equation}
\label{eq:dproj}
(\tfkm\rho(p_{0})\gd(p-\pp),\gd w)=0,\quad\feac w\in\Vms.
\end{equation}
The following lemma gives us the error bound of $\pp$ for the nonlinear parabolic problem.

\begin{lem}
\label{lem:main-eproj}
Let $p$ be the solution of \eqref{eq:weak-prob}. For each $t>0$, we define the elliptic projection $\pp\in\Vms$ that satisfies \eqref{eq:dproj}. Then,
\[
\|(\tfkm)^{1/2}\gd(p-\pp)(t)\|_{0}\preceq \Lambda^{-1/2}H\|(\tfkm)^{1/2}(q-\pt \rho(p))(t)\|_{0},
\]
where $\Lambda=\min_{1\leq i\leq N}\lambda^{(i)}_{L_{i}+1}$.
\end{lem}
\begin{proof}
Let $\pp\in\Vms$ be the projection of $p$. By boundedness of $\rho$ and orthogonality property, we can write
\begin{align*}
\|(\tfkm)^{1/2}\gd(p-\pp)\|^{2}_{0}\preceq\int_{\Om}(\tfkm)\rho(p_{0})|\gd(p-\pp)|^{2}d\x&=(\tfkm\rho(p_{0})\gd(p-\pp),\gd(p-\pp))\\
&=(\tfkm\rho(p_{0})\gd p,\gd(p-\pp)).
\end{align*}
Invoking again the boundedness of $\rho$, we have
\[|
(\tfkm\rho(p_{0})\gd p,\gd(p-\pp))|\preceq |(\tfkm\rho(p)\gd p,\gd(p-\pp))|.
\]
Now, from problem \eqref{eq:weak-prob}, we get that
\[
(\tfkm\rho(p)\gd p,\gd(p-\pp))=(q-\phi\pt\rho(p),\gd(p-\pp)),\quad\fall t>0.
\]
Therefore, we arrive at
\begin{equation}
\label{eq:lemprj1}
\|(\tfkm)^{1/2}\gd(p-\pp)\|^{2}_{0}\preceq(q-\phi\pt \rho(p),p-\pp)\leq \|\wk^{-1/2}(q-\phi\pt \rho(p))\|_{0}\|p-\pp\|_{s},\quad\fall t>0.
\end{equation}
Since $p-\pp\in\nV$, implies that $\pi(p-\pp)=0$. According to \cite{chung2018constraint}, the coarse blocks $\Ki$ with $i=1,\dots,\Nc$ are disjoint, so we obtain that $\pi_{i}(p-\pp)=0$, for all $i=1,2,\dots,\Nc$. Thus,  the local spectral problem \eqref{eq:lspec-prob} yields that
\begin{equation}
\label{eq:lemprj2}
\|p-\pp\|^{2}_{s}=\sum_{i=1}^{\Nc}\|p-\pp\|^{2}_{s_{i}}=\sum_{i=1}^{\Nc}\|(I-\pi_{i})(p-\pp)\|^{2}_{s_{i}}\preceq \frac{1}{\Lambda}\sum_{i=1}^{\Nc}\|(\tfkm)^{1/2}\gd(p-\pp)\|^{2}_{0,\Ki},
\end{equation}
where $\Lambda=\min_{1\leq i\leq N}\lambda^{(i)}_{L_{i}+1}$. Therefore, by combining \eqref{eq:lemprj1} and \eqref{eq:lemprj2}, and using the fact $|\gd\chi_{i}|=\Or(H^{-1})$, we obtain
\[
\|(\tfkm)^{1/2}\gd (p-\pp)\|_{0}\preceq \Lambda^{-1/2}H\|\k^{-1/2}(q-\phi\pt \rho(p))\|_{0}.
\]
This completes the proof.
\end{proof}

The above estimate is the essence of the following result.

\begin{lem}
\label{lem:cor}
Under Assumptions of Lemma \ref{lem:main-eproj}, the following estimates hold
\begin{align*}
\|p-\pp\|_{0}&\preceq \Lambda^{-1}H^{2}\|\k^{-1/2}(q-\phi\pt\rho(p))\|_{0},\\
\|\pt(p-\pp)\|_{0}&\preceq \Lambda^{-1}H^{2}\|\k^{-1/2}\pt(q-\phi\pt\rho(p))\|_{0}.
\end{align*}
\end{lem}

\begin{proof}
First, we will invoke the duality argument. For each $t>0$, we define $w\in\Vz$ by 
\[
\int_{\Om}\k\rho(p_{0})\gd w\cdot\gd vd\x=\int_{\Om}(p-\pp)vd\x,\quad\feac v\in\Vz,
\]
and consider $\pw$ as elliptic projection of $w$ in $\Vms$. By Lemma \ref{lem:main-eproj}, for $v=p-\pp$, we have
\begin{align*}
\|p-\pp\|^{2}_{0}=\int_{\Om}\k\rho(p_{0})\gd w\cdot\gd (p-\pp)d\x &= \int_{\Om}\k\rho(p_{0})\gd (w-\pw)\cdot\gd (p-\pp)d\x\\
&\preceq \int_{\Om}\tfkm\rho(p_{0})\gd (w-\pw)\cdot\gd (p-\pp)d\x\\
&\leq \|(\tfkm)^{1/2}\gd(w-\pw)\|_{0}\|(\tfkm)^{1/2}\gd(p-\pp)\|_{0}\\
&\preceq \left(H\Lambda^{-1/2}\max\{\k^{-1/2}\}\|p-\pp\|_{0}\right)\\
&\quad\times\left(H\Lambda^{-1/2}\|\k^{-1/2}(q-\phi\pt\rho(p))\|_{0}\right).
\end{align*}
Hence, we have 
\[
\|p-\pp\|_{0}\preceq \Lambda^{-1}H^{2}\|\k^{-1/2}(q-\phi\pt\rho(p))\|_{0}.
\]
By a similar computation, we can obtain the second estimate. This completes the proof.
\end{proof}

We will derive an error estimate for the difference between the solution of \eqref{eq:main-prob} and the CEM-GMsFEM solution of \eqref{eq:sms-prob} using the framework of \cite{fu2022generalized}. 

\begin{thm}
Let $p$ be the solution obtained from \eqref{eq:weak-prob}, $\pms\in\Vms$ be the multiscale solution of \eqref{eq:sms-prob} using CEM-GMsFEM and $\pp$ be an elliptic projection of $p$ in $\Vms$. Then, the following error estimate holds
\begin{align*}
\|(p-\pms)(t)\|_{0}^{2}+\int_{0}^{T}\|(\tfkm)^{1/2}\gd(p-\pms)\|^{2}_{0}dt&\preceq\Lambda^{-1}H^{2}\left(\|\k^{-1/2}\pt(q-\phi\pt \rho(p))(t)\|_{0}^{2}\right.\\
&\quad+\left.\int_{0}^{T}\|\k^{-1/2}\pt(q-\phi\pt\rho(p))\|^{2}_{0}dt\right)\\
&\quad+\|(\pp-\pms)(0)\|^{2}_{0}.
\end{align*}
\end{thm}

\begin{proof}
Subtracting \eqref{eq:sms-prob} from \eqref{eq:weak-prob}, and using \eqref{eq:res-prob}, we have that
\[
(\phi\pt\rho(p),v)+\left(\tfkm\rho(p)\gd p,\gd v\right)-(\phi\pt\rho(\pms),v)-\left(\tfkm\rho(\pms)\gd\pms,\gd v\right)= 0,\quad \feac v\in\Vms.
\]
Since $\pp\in\Vms$, we put $v=\pp-\pms$, then follows that
\begin{equation}
\label{eq:01lem2}
\underset{I_{1}}{\underbrace{(\phi\pt(\rho(p)-\rho(\pms)),\pp-\pms)}}+\underset{I_{2}}{\underbrace{\left(\tfkm(\rho(p)\gd p-\rho(\pms)\gd\pms),\gd (\pp-\pms)\right)}}=0.
\end{equation}
About $I_{1}$, we can rewrite  
\begin{equation}
\label{eq:012lem2}
I_{1}=\underset{I_{3}}{\underbrace{(\phi\pt(\rho(\pp)-\rho(\pms)),\pp-\pms)}}+\underset{I_{4}}{\underbrace{(\phi\pt(\rho(p)-\rho(\pp)),\pp-\pms)}}.
\end{equation}
For $I_{3}$, we obtain
\begin{align*}
I_{3}&=\frac{d}{dt}\int_{\Om}\phi\int_{0}^{\pp-\pms}\rho'(\pp+\xi)\xi d\xi d\x-\underset{I_{5}}{\underbrace{\int_{\Om}\phi\int_{0}^{\pp-\pms}\rho''(\pp+\xi)\pt\pp\xi d\xi d\x}}\\
&\quad+\underset{I_{6}}{\underbrace{\int_{\Om}\phi\int_{\Om}(\rho'(\pms)-\rho'(\pp))\pt\pp(\pp-\pms)d\x}}.
\end{align*}
Following \cite{park2005mixed,kim2007multiscale}, we have that the terms $I_{5}$ and $I_{6}$ are bounded by $\|\pp-\pms\|^{2}_{0}$.
We deduce that 
\[
I_{3}\geq \frac{d}{dt}\int_{\Om}\phi\int_{0}^{\pp-\pms}\rho'(\pp+\xi)\xi d\xi d\x-C_{1}\|\pp-\pms\|^{2}_{0},
\]
where $C_{1}$ is a positive constant independent of the mesh size. In virtue of $\rho'$ being bounded below positively, we have
\[
\int_{\Om}\phi\int_{0}^{\pp-\pms}\rho'(\pp+\xi)\xi d\xi d\x\geq C_{2}\|\pp-\pms\|^{2}_{0}.
\]
Then, for $I_{3}$, we obtain
\begin{equation}
\label{eq:0121lem02}
I_{3}=(\phi\pt(\rho(\pp)-\rho(\pms)),\pp-\pms)\geq\frac{d}{dt}\|\pp-\pms\|^{2}_{0}-C_{3}\|\pp-\pms\|^{2}_{0}.
\end{equation}
For $I_{4}$, by using the chain rule and Young's inequality, one can get
\begin{align*}
I_{4}&=(\phi(\rho'(p)-\rho'(\pp))\pt\pp,\pp-\pms)+(\phi\rho'(p)(\pt p-\pt\pp),\pp-\pms)\\
&\preceq \|p-\pp\|^{2}_{0}+\|\pt(p-\pp)\|^{2}_{0}+\|\pp-\pms\|^{2}_{0}.
\end{align*}
Now, for $I_{2}$, we get
\[
I_{2}=\underset{I_{7}}{\underbrace{\left(\tfkm(\rho(\pms)\gd(\pp-\pms),\gd (\pp-\pms)\right)}}+\underset{I_{8}}{\underbrace{\left(\tfkm\left(\rho(p)\gd p-\rho(\pms)\gd\pp\right),\gd (\pp-\pms)\right)}}
\]
Then, for $I_{7}$ we have
\[
I_{7}\geq C\|(\tfkm)^{1/2}\gd(\pp-\pms)\|^{2}_{0},
\]
and for $I_{8}$ by invoking Young's inequality, one obtains
\begin{align*}
I_{8}&=\left(\tfkm(\rho(p)-\rho(\pms))\gd p,\gd (\pp-\pms)\right)+\left(\tfkm(\rho(\pms)\gd p-\rho(\pms)\gd\pp),\gd (\pp-\pms)\right)\\
&\leq C\left( \|(\tfkm)^{1/2}\gd(p-\pp)\|^{2}_{0}+\|p-\pms\|^{2}_{0}\right)+\epsilon\|(\tfkm)^{1/2}\gd(\pp-\pms)\|^{2}_{0},
\end{align*}
where in the last inequality we use the boundedness of $\rho$ and $\rho'$, and $p\in W^{1,\infty}(\Om)$.
Combining the above estimates and taking $\epsilon$ small enough, we can obtain
\begin{align*}
\frac{d}{dt}\|\pp-\pms\|^{2}_{0}+\|(\tfkm)^{1/2}\gd(\pms-\pp)\|^{2}_{0}&\preceq\|(\tfkm)^{1/2}\gd(p-\pp)\|^{2}_{0}+\|p-\pp\|^{2}_{0}+\|\pt(p-\pp)\|^{2}_{0}\\
&\quad+\|\pp-\pms\|^{2}_{0}.
\end{align*}
Integrating with respect to time $t$ and invoking the continuous Gronwall's inequality \cite{cannon1999modified}, we can infer that
\begin{align*}
\|(\pp-\pms)(t)\|^{2}_{0}+\int_{0}^{T}\|(\tfkm)^{1/2}\gd(\pms-\pp)\|^{2}_{0}dt&\preceq\|(\pp-\pms)(0)\|^{2}_{0}+\int_{0}^{T}\|(\tfkm)^{1/2}\gd(p-\pp)\|^{2}_{0}dt\\
&\quad+\int_{0}^{T}\left(\|p-\pp\|^{2}_{0}+\|\pt(p-\pp)\|^{2}_{0}\right)dt.
\end{align*}
Thus, we use the triangle inequality to get
\begin{align*}
\|(p-\pms)(t)\|^{2}_{0}+\int_{0}^{T}\|(\tfkm)^{1/2}\gd(p-\pms)\|^{2}_{0}dt&\preceq\|(\pp-\pms)(0)\|^{2}_{0}+\int_{0}^{T}(\|(\tfkm)^{1/2}\gd(p-\pp)\|^{2}_{0}dt\\
&\quad+\int_{0}^{T}\left(\|p-\pp\|^{2}_{0}+\|\pt(p-\pp)\|^{2}_{0}\right)dt+ \|(p-\pp)(t)\|^{2}_{0}.
\end{align*}
Finally, the proof is completed by using Lemmas \ref{lem:main-eproj} and \ref{lem:cor}.
\end{proof}

\subsection{A posteriori error estimate}
We shall give an \emph{a posteriori} error estimate, which provides a computable error bound to access the quality of the numerical solution. To begin, notice that since $\Vms\subset\Vh$, we can derive from the fully-discrete approximation a residual expression defined by
\begin{equation}
\label{eq:res-prob}
\Rmsn(v):=(\phi\rho(\pmsn),v)-(\phi\rho(\pmsnm),v)+\dtn\left(\tfkm\rho(\pmsn)\gd\pmsn,\gd v\right)-\dtn(q,v),\quad\feac v\in\Vms.
\end{equation}
We also consider, the local residuals. For each coarse node $x_{i}$, we define $\oi$ be the set of coarse blocks having the vertex $x_{i}$. For each coarse neighborhood $\oi$, we define the local residual functional $r_{i}:\V\to\R$ by
\[
\rni(v)=r(\chi_{i} v)=(\phi\rho(\pmsn),\chi_{i} v)-(\phi\rho(\pmsnm),\chi_{i} v)+\dtn\left(\tfkm\rho(\pmsn)\gd\pmsn,\gd \chi_{i} v\right)-\dtn(q,\chi_{i} v),
\]
for all $v\in\V$. The local residual $r_{i}$ gives a measure of the error $p-\pmsn$ in the coarse neighborhood $\oi$.

\begin{thm}
Let $p_{n}$ be the solution obtained from \eqref{eq:main-prob} at $t_{n}$ and $\pms_{n}\in\Vms$ denote the CEM-GMsFEM solution of the fully discrete scheme of \eqref{eq:fms-prob} at $t_{n}$. Then, there exists a positive constant $C$ independent of the mesh size such that
\[
\|p_{\Nt}-\pms_{\Nt}\|^{2}_{0}+\sum_{n=1}^{\Nt}\dtn\|(\tfkm)^{1/2}\gd(\pn-\pmsn)\|^{2}_{0}\preceq (1+\Lambda^{-1})\sum_{n=1}^{\Nt}\sum_{i=1}^{\Nc}\|\widetilde{r}^{n}_{i}\|^{2}_{\V^{*}_{i}}+\|p_{0}-\pms_{0}\|^{2}_{0},
\]
where
\[
\wrni=\dtn\int_{\oi}q_{n}vd\x-\int_{\oi}\phi(\rho(\pms_{n})-\rho(\pms_{n-1}))vd\x - \dtn\int_{\oi}\tfkm\rho(\pms_{n})\gd\pms_{n}\cdot\gd vd\x,
\]
and the residual norm is defined by
\[
\|\wrni\|_{\V^{*}_{i}}=\sup_{v\in\Lt(t_{n},t_{n+1};\Hoz(\oi))}\frac{\wrni}{\|v\|_{\V_{i}}}.
\]
\end{thm}
\begin{proof}
Subtracting \eqref{eq:fms-prob} from \eqref{eq:fweak-prob}, we get for $p\in\V$ at $t_{n}$
\begin{equation}
\label{eq:pos01}
(\phi(\rho(\pn)-\rho(\pmsn)),v)-(\phi(\rho(\pnm)-\rho(\pmsnm),v)+\dtn\left(\tfkm(\rho(\pn)\gd\pn-\rho(\pmsn)\gd\pmsn,\gd v\right)=0,
\end{equation}
for each $v\in\Vms$ Putting $v = \pn-\pmsn$ and using the fact that $\rho$ is bounded below positively, we easily obtain that
\[
(\phi(\rho(\pn)-\rho(\pmsn)),\pn-\pmsn)\geq C\|\pn-\pmsn\|^{2}_{0}.
\]
Similarly, for the second term of \eqref{eq:pos01}, we can use the boundedness of $\rho$ and the Young's inequality to yield
\[
(\phi(\rho(\pnm)-\rho(\pmsnm),\pn-\pmsn)\leq C\|\pnm-\pmsnm\|^{2}_{0}+\epsilon\|\pn-\pmsn\|^{2}_{0}.
\] 
Gathering the above inequalities and for $\epsilon$ small enough, we arrive to 
\[
\begin{split}
(\phi(\rho(\pn)-\rho(\pmsn)),\pn-\pmsn)-(\phi(\rho(\pnm)-\rho(\pmsnm),\pn-\pmsn)&\geq C\left(\|\pn-\pmsn\|^{2}_{0}\right.\\
&\quad-\left.\|\pnm-\pmsnm\|^{2}_{0}\right).
\end{split}
\]
For third term of \eqref{eq:pos01}, we have that
\[
\left(\tfkm\left(\rho(\pn)\gd\pn-\rho(\pmsn)\gd\pmsn\right),\gd (\pn-\pmsn)\right)\geq C\|(\tfkm)^{1/2}\gd(\pn-\pmsn)\|_{0}^{2}.
\]
Thus, these above inequalities drive us the expression
\begin{align*}
\|\pn-\pmsn\|^{2}_{0}&-\|\pnm-\pmsnm\|^{2}_{0}+\dtn\|(\tfkm)^{1/2}\gd(\pn-\pmsn)\|^{2}_{0}\preceq (\phi(\rho(\pn)-\rho(\pmsn)),\pn-\pmsn)\\
& \quad-(\phi(\rho(\pnm)-\rho(\pmsnm)),\pn-\pmsn)\\
& \quad+\dtn\left(\tfkm\left(\rho(\pn)\gd\pn-\rho(\pmsn)\gd\pmsn\right),\gd (\pn-\pmsn)\right)\\
&=(\phi(\rho(\pn)-\rho(\pmsn)),\pn-\pmsn)-(\phi(\rho(\pnm)-\rho(\pmsnm)),\pn-\pmsn)\\
& \quad+\dtn\left(\tfkm\left(\rho(\pn)\gd\pn-\rho(\pn)\gd\pmsn\right),\gd (\pn-\pmsn)\right)\\
& \quad+\dtn\left(\tfkm\left(\rho(\pn)\gd\pmsn-\rho(\pmsn)\gd\pmsn\right),\gd (\pn-\pmsn)\right).
\end{align*}
Reorganizing the terms and using the definition of the weak formulation \eqref{eq:main-prob}, we get that
\begin{equation}
\label{eq:pos02}
\begin{split}
&\|\pn-\pmsn\|^{2}_{0}-\|\pnm-\pmsnm\|^{2}_{0}+\dtn\|(\tfkm)^{1/2}\gd(\pn-\pmsn)\|^{2}_{0}\preceq\dtn(q^{n},\pn-\pmsn)\\
&\quad-(\phi(\rho(\pmsn)-\rho(\pmsnm)),\pn-\pmsn)-\dtn\left(\tfkm\rho(\pn)\gd\pmsn,\gd (\pn-\pmsn)\right)\\
&\quad+\dtn\left(\tfkm\left(\rho(\pn)-\rho(\pmsn)\right)\gd\pmsn,\gd (\pn-\pmsn)\right).
\end{split}
\end{equation}
We will limit the right-hand side of \eqref{eq:pos02}. In light of residual expression \eqref{eq:res-prob}, we have that
\[
\Rms(v)=0,\quad\feac v\in\Vms.
\]
Denote $w=\pmsn-\pn$ and use $\pw\in\Vms$, the elliptic projection of $w$. Thus,
\begin{align*}
\Rmsn(w)=\Rmsn(w-\pw)&=\dtn(q_{n},w-\pw)-(\phi(\rho(\pmsn)-\rho(\pmsnm)),w-\pw)\\
&\quad-\dtn\left(\tfkm\rho(\pmsn)\gd\pmsn,\gd (w-\pw)\right).
\end{align*}
Let us rewrite $\Rmsn(w-\pw)=\sum_{i=1}^{\Nc}\wrni(\chi_{i}(w-\pw))$ \cite{chung2018fast,li2019constraint}, then
\begin{equation}
\label{eq:main-rin}
\begin{split}
\sum_{i=1}^{\Nc}\wrni(\chi_{i}(w-\pw))&=\dtn\sum_{i=1}^{\Nc}\left(\int_{\oi}q_{n}\chi_{i}(w-\pw)d\x-\int_{\oi}\phi\frac{\rho(\pmsn)-\rho(\pmsnm)}{\dtn}\chi_{i}(w-\pw)d\x\right.\\
&\quad-\left.\int_{\oi}\tfkm\rho(\pmsn)\gd\pmsn\cdot\gd\chi_{i}(w-\pw)d\x\right).
\end{split}
\end{equation}
Note that
\begin{equation}
\label{eq:ps01}
\sum_{i=1}^{\Nc}\widetilde{r}^{n}_{i}(\chi_{i}(w-\pw))\leq \sum_{i=1}^{\Nc}\|\widetilde{r}^{n}_{i}\|_{\V^{*}_{i}}\|\chi_{i}(w-\pw)\|_{\V_{i}},
\end{equation}
where $\|v\|_{\V_{i}}=\|v\|^{2}_{0,\oi}+\dtn\|(\tfkm)^{1/2}\gd v\|^{2}_{0,\oi}$, where $\|\cdot\|_{0,\oi}$ denotes the $\L^{2}$-norm restricted to $\oi$. Notice also that,
\begin{equation}
\label{eq:ps02}
\|(\tfkm)^{1/2}\gd\chi_{i}(w-\pw)\|_{0,\oi}\preceq\left(\|w-\pw\|^{2}_{s(\oi)}+\|(\tfkm)^{1/2}\gd(w-\pw)\|^{2}_{0,\oi}\right)^{1/2},
\end{equation}
where $\|\cdot\|_{s(\oi)}$ represents the $s$-norm restricted to $\oi$. For the second term on the right-hand side, by using the orthogonality property, \ie, $((\tfkm)^{1/2}\rho(p_{0})\gd(w-\pw),\gd v)=0$, for all $v\in\Vms$, we get
\begin{equation}
\label{eq:ps03}
\|(\tfkm)^{1/2}\gd(w-\pw)\|^{2}_{0,\oi}\preceq \|(\tfkm)^{1/2}\gd w\|^{2}_{0,\oi}.
\end{equation}
Now, for the first term on the right-hand side of \eqref{eq:ps02}, we shall use the duality argument. Let $g=\widetilde{\k}(w-\pw)$ and $z\in\Vz$ be the solution of problem below 
\[
\int_{\oi}\k\rho(p_{0})\gd z\cdot\gd vd\x=\int_{\oi}gvd\x,\quad\feac v\in\Vz.
\]
Putting $v=w-\pw$, using the Cauchy-Schwarz inequality and equation \eqref{eq:ps03}, we arrives to
\begin{align*}
\|w-\pw\|^{2}_{s(\oi)}&=\int_{\oi}g(w-\pw)d\x=\int_{\oi}\k\rho(p_{0})\gd z\cdot\gd (w-\pw)d\x\\
&=\int_{\oi}\k\rho(p_{0})\gd (z-\hat{z})\cdot\gd (w-\pw)d\x\\
&\leq \|(\k\rho(p_{0}))^{1/2}\gd(z-\hat{z})\|_{0,\oi}\|(\k\rho(p_{0}))^{1/2}\gd(w-\pw)\|_{0,\oi}\\
&\preceq \Lambda^{-1/2}\|\widetilde{\k}^{-1/2}g\|_{\L^{2}(\oi)}\|(\tfkm)^{1/2}\gd(w-\pw)\|_{0,\oi}\\
&\preceq \Lambda^{-1/2}\|w-\pw\|_{s(\oi)}\|(\tfkm)^{1/2}\gd w\|_{0,\oi}.
\end{align*}
So, we have
\begin{equation}
\label{eq:ps04}
\|w-\pw\|_{s(\oi)}\preceq \Lambda^{-1/2}\|(\tfkm)^{1/2}\gd w\|_{0,\oi}.
\end{equation}
Gathering \eqref{eq:ps02}-\eqref{eq:ps04}, we arrive to
\[
\|(\tfkm)^{1/2}\gd\chi_{i}(w-\pw)\|_{0,\oi}\preceq(1+\Lambda^{-1})^{1/2}\|(\tfkm)^{1/2}\gd w\|_{0,\oi}.
\]
Analogously, we estimate $\|\chi_{i}(w-\pw)\|_{0,\oi}\preceq\|(\tfkm)^{1/2}\gd w\|_{0,\oi}$. Therefore,
\begin{equation}
\label{eq:pos02-1}
\sum_{i=1}^{\Nc}\|\wrni\|_{\V^{*}_{i}}\|\chi_{i}(w-\pw)\|_{\V_{i}}\preceq(1+\Lambda^{-1})^{1/2}\sum_{i=1}^{\Nc}\|\wrni\|_{\V^{*}_{i}}\|(\tfkm)^{1/2}\gd w\|_{0,\oi}.
\end{equation}
For the last term of the right-hand side of \eqref{eq:pos02}, we have
\begin{equation}
\label{eq:pos02-2}
\dtn\left(\tfkm\left(\rho(\pn)-\rho(\pmsn)\right)\gd\pmsn,\gd (\pn-\pmsn)\right)\preceq \dtn\|\rho(\pn)-\rho(\pmsn)\|_{0}\|(\tfkm)^{1/2}\gd(\pn-\pmsn)\|_{0}.
\end{equation}
Combining \eqref{eq:pos02-1} and \eqref{eq:pos02-2}, and using Young's inequality, one can express for \eqref{eq:pos02} by summing over all $n$
\begin{align*}
\|p_{\Nt}-\pms_{\Nt}\|^{2}_{0}+\sum_{n=1}^{\Nt}\dtn\|(\tfkm)^{1/2}\gd(\pn-\pmsn)\|^{2}_{0}&\preceq (1+\Lambda^{-1})\sum_{n=1}^{\Nt}\sum_{i=1}^{\Nc}\|\wrni\|^{2}_{\V^{*}_{i}}\\
&\quad+\sum_{n=1}^{\Nt}\dtn\|\rho(\pn)-\rho(\pmsn)\|_{0}+\|p_{0}-\pms_{0}\|^{2}_{0}.
\end{align*}
The proof is completed by using the discrete Gronwall inequality.
\end{proof}

\section{Numerical results}
\label{sec:numerical}

We now present numerical results of the nonlinear single-phase compressible flow in highly heterogeneous porous media with the performance of the CEM-GMsFEM that are summarized in the following three separate experiments. All parameters on the flow model, boundary, and initial conditions in each numerical experiment are described in detail to allow their proper reproduction of them. The main aim of the simulation is to demonstrate the viability of the proposed numerical approximation and improve the convergence rate shown in Section \ref{sec:convergence}. We implement the CEM-GMsFEM in \texttt{Matlab} language and use the numerical experiments presented in \cite{wang2014algebraic,fu2022generalized} as a reference guide to our three-dimensional experiments. We will use the Euler-backward for the time discretization and a Newton-Raphson method with a tolerance of $10^{-6}$ for the non-linear problem. Only $2-4$ Newton's iterations are needed in the computations presented below. 
\begin{figure}[!h]
\centering
\subfloat[Perm1][$\k_{1}.$]{\includegraphics[width=0.50\textwidth]{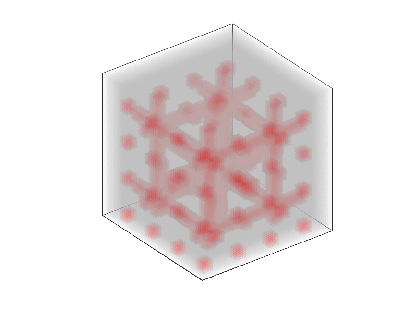}\label{fig:perm1}}
\subfloat[perm1][$\k_{2}.$]{\includegraphics[width=0.50\textwidth]{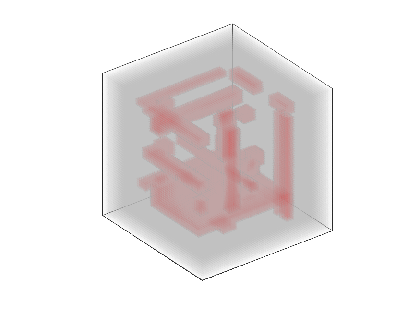}\label{fig:perm2}}\\
\subfloat[Perm3][$\k_{3}.$]{\includegraphics[width=0.50\textwidth]{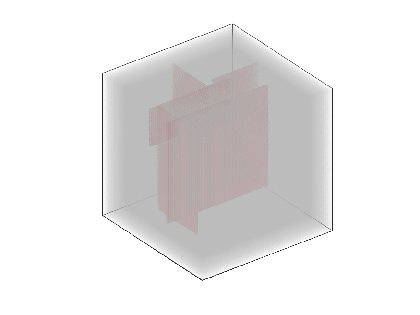}\label{fig:perm3}}
\subfloat[Perm4][$\k_{4}$ in logarithm scale.\label{fig:perm4}]{\includegraphics[width=0.45\textwidth]{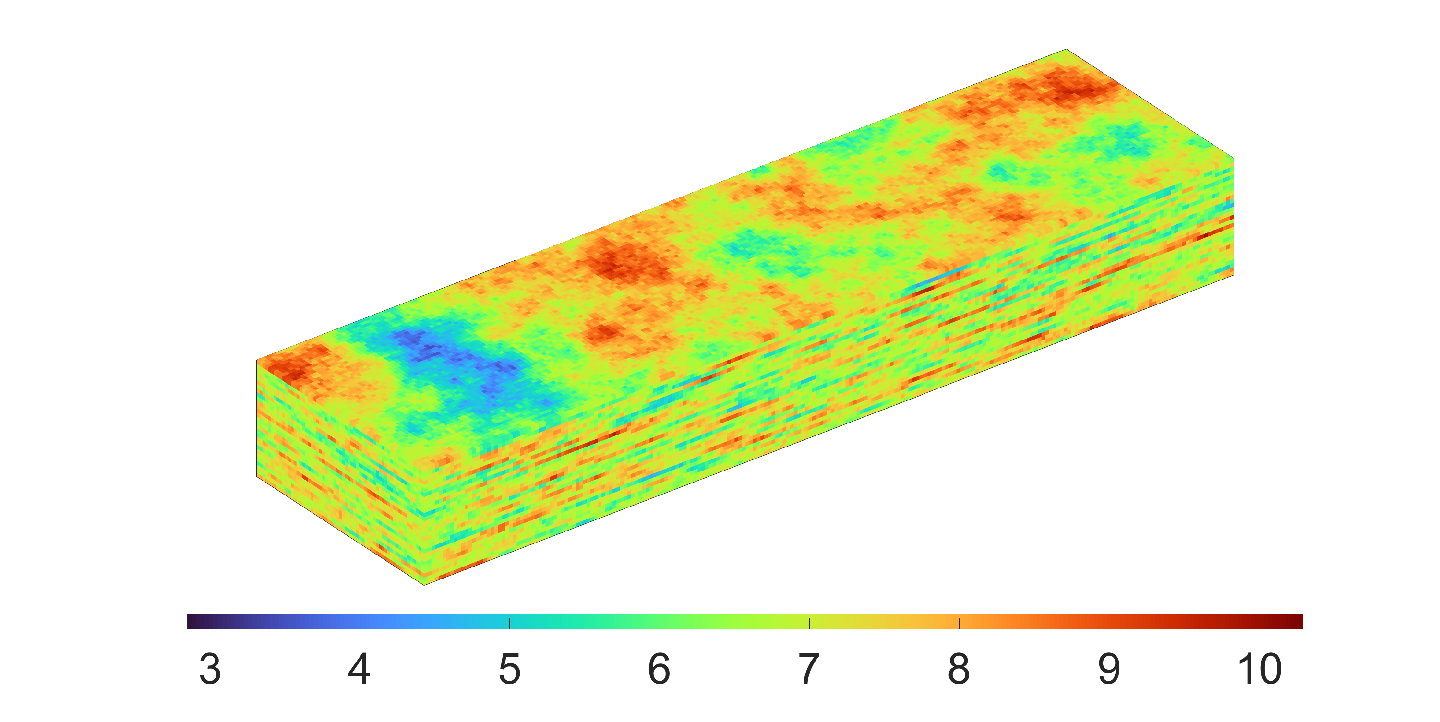}}

\caption{\label{fig:perm-fields} Different permeability fields. }
\end{figure}

We consider three high-contrast permeability fields that are the disjoint union of a background region with $10^{5}$ millidarcys and other regions of $10^{9}$ millidarcys (see Figure \ref{fig:perm-fields}). We also consider a fractured porous medium. In this case, the permeability value in fractures is much larger than in the surrounding medium. Finally, we employ the first $30$ layers of SPE10 $3$D dataset from \cite{christie2001spe}, which is widely used in the reservoir simulation community to test multiscale approaches. All experiments employ parameters of viscosity $\mu=5$cP, porosity $\phi=500$, fluid compressibility $c=1.0\times10^{-8}\text{Pa}^{-1}$, the reference pressure $\pref=2.00\times10^{7}$Pa, and the reference density $\rref=850$kg/$\mbox{m}^{3}$.

\begin{example}
\label{example:1}
For our first example, we set a fine grid resolution of $64^{3}$, with a size of $h=20$m, and a coarse grid resolution of $8^{3}$ of size $H=8h$. The coarsening factor is chosen due this coarse grid provides the most computationally efficient performance for the method. For the CEM-GMsFEM, we use $4$ basis functions and $4$ oversampling layers. We know well that the number of bases is sufficient to improve the accuracy of CEM-GMsFEM; see \cite{chung2018constraint}. Then, we have a coarse system with dimension $4916 \,(=729\times\mbox{number of basis functions})$, and the fine-scale system has a dimension of $274625$. The permeability field $\k_{1}$ used in this experiment is depicted in Figure~\ref{fig:perm1}. We define a model configuration as follows: four vertical injectors in each corner and a unit sink in the center of the domain to drive the flow, and employ the full zero Neumann boundary condition and an initial pressure field $p_{0}=2.16\times 10^{7}$Pa. We consider a fine grid resolution of $64^{3}$, whose fine grid size is given by $h=20$m, meanwhile the coarse grid resolution of $8^{3}$, whose coarse size is given by $H=8h$. The time step $\dtn$ is $7$ days, and the total time simulation will be $T=25\dtn(=175\mbox{ days})$. Figure~\ref{fig:example1} shows the pressure profiles with the sink term and zero Neumann boundary conditions in Fig.~\ref{fig:example1} for different instants at day $t=77$ and $t=147$. In this case, we obtain a relative $L^{2}$- and  $H^{1}$-error of 2.1138E-03 and 3.8058E-02 respectively.
\begin{figure}[!h]
\centering
\subfloat[example1][$t=77$.]{\includegraphics[width=0.7\textwidth]{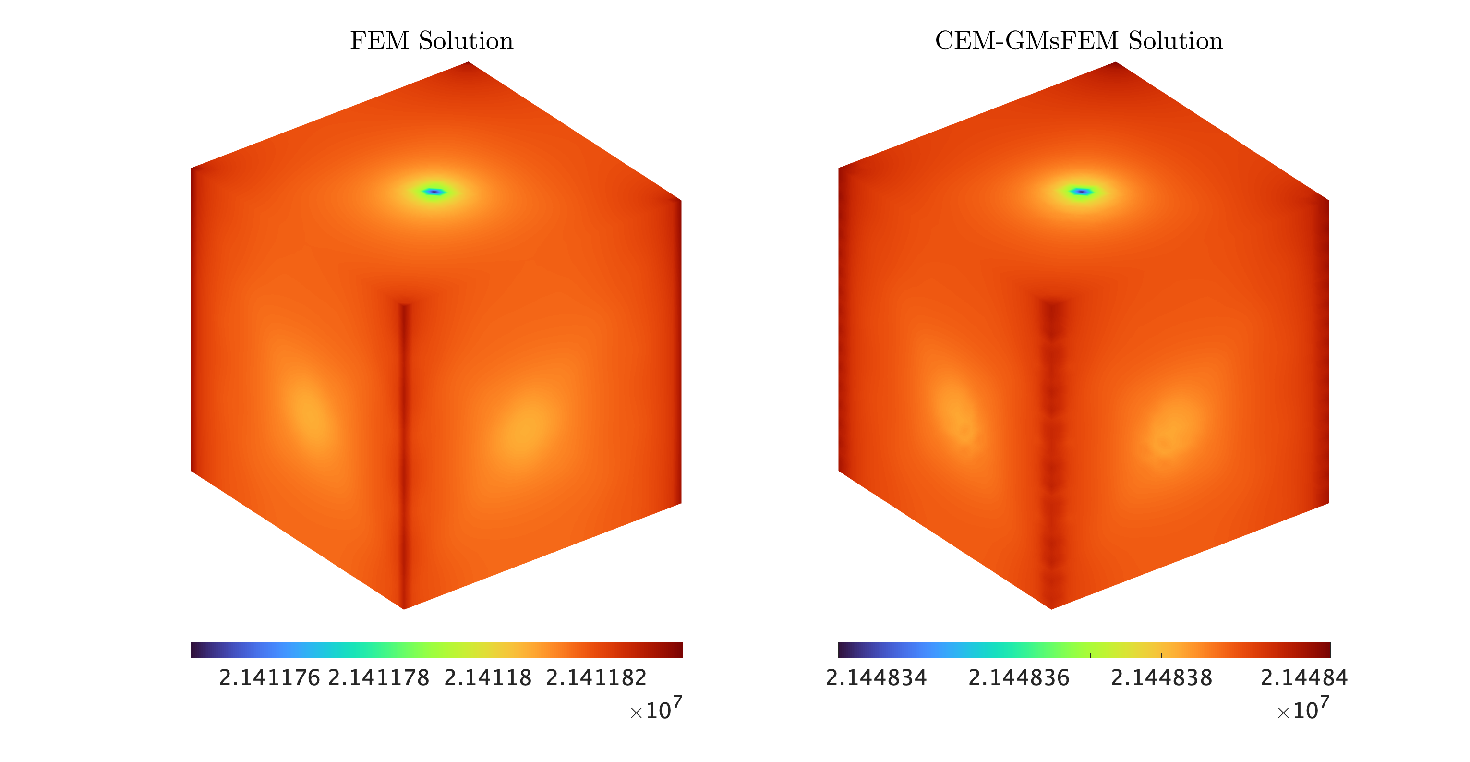}}\\
\subfloat[example1][$t=147$.]{\includegraphics[width=0.7\textwidth]{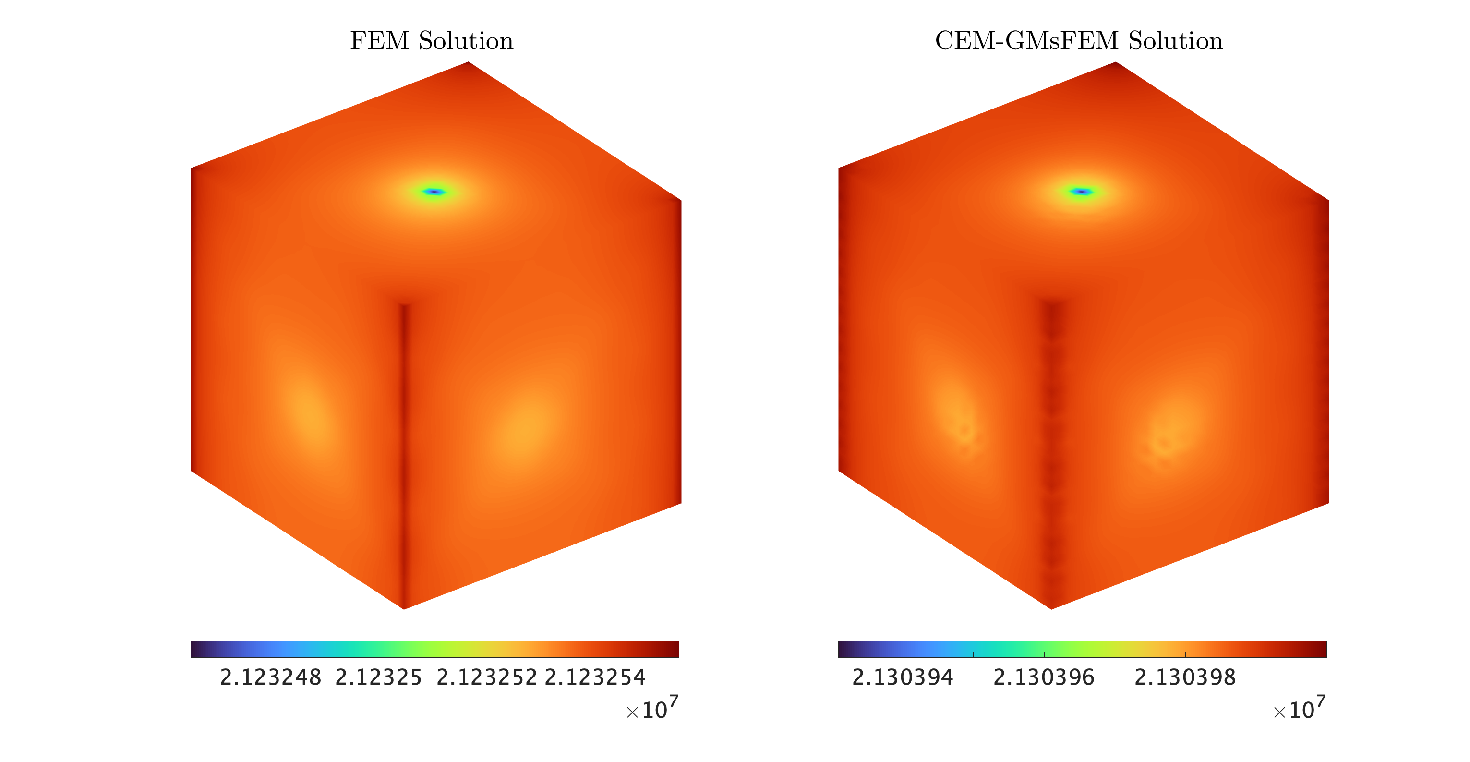}}

\caption{\label{fig:example1} Numerical solution of Example~\ref{example:1} using full-zero Neumann boundary conditions and high-contrast permeability field $\k_{1}$, see Figure~\ref{fig:perm1}. The fine-scale reference solution (left) and CEM-GMsFEM solution (right) with $4$ basis function and $4$ oversampling layers at (a) $t=77$ and (b) $t=147$.}
\end{figure}
\end{example}

\begin{example}
\label{example:2}
We consider a combination of zero Neumann and nonzero Dirichlet boundary conditions as in \cite{wang2014algebraic,fu2022generalized}. We set a fine grid resolution of $64^{3}$, with a size of $h=20$m, and different coarse grid resolutions of $4^{3},8^{3}$ and $16^{3}$. The time step $\dtn$ and total time simulation are the same as Example~\ref{example:1}. We impose zero Neumann condition on boundaries of planes $xy$ and $xz$ and let $p=2.16\times10^{7}$Pa in the first $yz$ plane and $p=2.00\times10^{7}$Pa in the last $yz$ plane for all time instants, no additional source is imposed.  The permeability field used is $\k_{2}$ (Figure~\ref{fig:perm2}). The pressure difference will drive the flow, and the initial field $p_{0}$ linearly decreases along the $x$ axis and is fixed in the $yz$ plane. Table~\ref{tab:example2} shows that numerical results use $4$ basis functions on each coarse block with different coarse grid sizes ($H=4h,8h$ and $16h$), where $\varepsilon_{0}$ and $\varepsilon_{1}$ denote the relative $L^{2}$ and energy error estimate between the reference solution and CEM-GMsFEM solution defined by
\[
\varepsilon_{0}=\left(\frac{\sum_{i=1}^{\Nt}(\ph_{i}-\pms_{i})^{2}}{\sum_{i=1}^{\Nt}(\ph_{i})^{2}}\right)^{1/2},\quad\varepsilon_{1}=\left(\frac{\sum_{i=1}^{\Nt}(\fkm)^{1/2}\rho(p_{0})\gd(\ph_{i}-\pms_{i})^{2}}{\sum_{i=1}^{\Nt}(\ph_{i})^{2}}\right)^{1/2},
\]
where $\ph_{i}$ denotes the references solution and $\pms_{i}$ is the CEM-GMsFEM approximation for $i=1,\dots,\Nt$. For instance, for a coarse grid size of $H=8h$, we obtain the relative errors $\varepsilon_{0}=$ 4.5581E-04 and $\varepsilon_{1}=$ 2.5249E-01. In Figure~\ref{fig:example2}, we depict the numerical solution profiles with a fine grid resolution of $64^{3}$ and coarse grid resolution of $8^{3}$ at day $t=105$ and $t=140$, which is hard to find any difference between the reference solution and CEM-GMsFEM solution. Therefore, we have a good agreement.
\begin{table}
\centering
\begin{tabular}{|c|c|c|c|c|}
\hline
Number of basis & $H$ &  Number of oversampling layers $m$ & $\varepsilon_{0}$ & $\varepsilon_{1}$\\
\hline
 4 & $4h$ & $3$ & 2.4004E-03 & 4.4227E-01\\
 4 & $8h$ & $4$ & 4.5581E-04 & 2.5249E-01\\
 4 & $16h$ & $5$ & 1.4257E-04 & 1.258E-01\\
\hline
\end{tabular}
\caption{\label{tab:example2} Convergence rate for Example~\ref{example:2} with different numbers of oversampling layers ($m$) with a combination of zero Neumann and nonzero Dirichlet boundary conditions.}
\end{table}

\begin{figure}[!h]
\centering
\subfloat[example2][$t=105$.]{\includegraphics[width=0.8\textwidth]{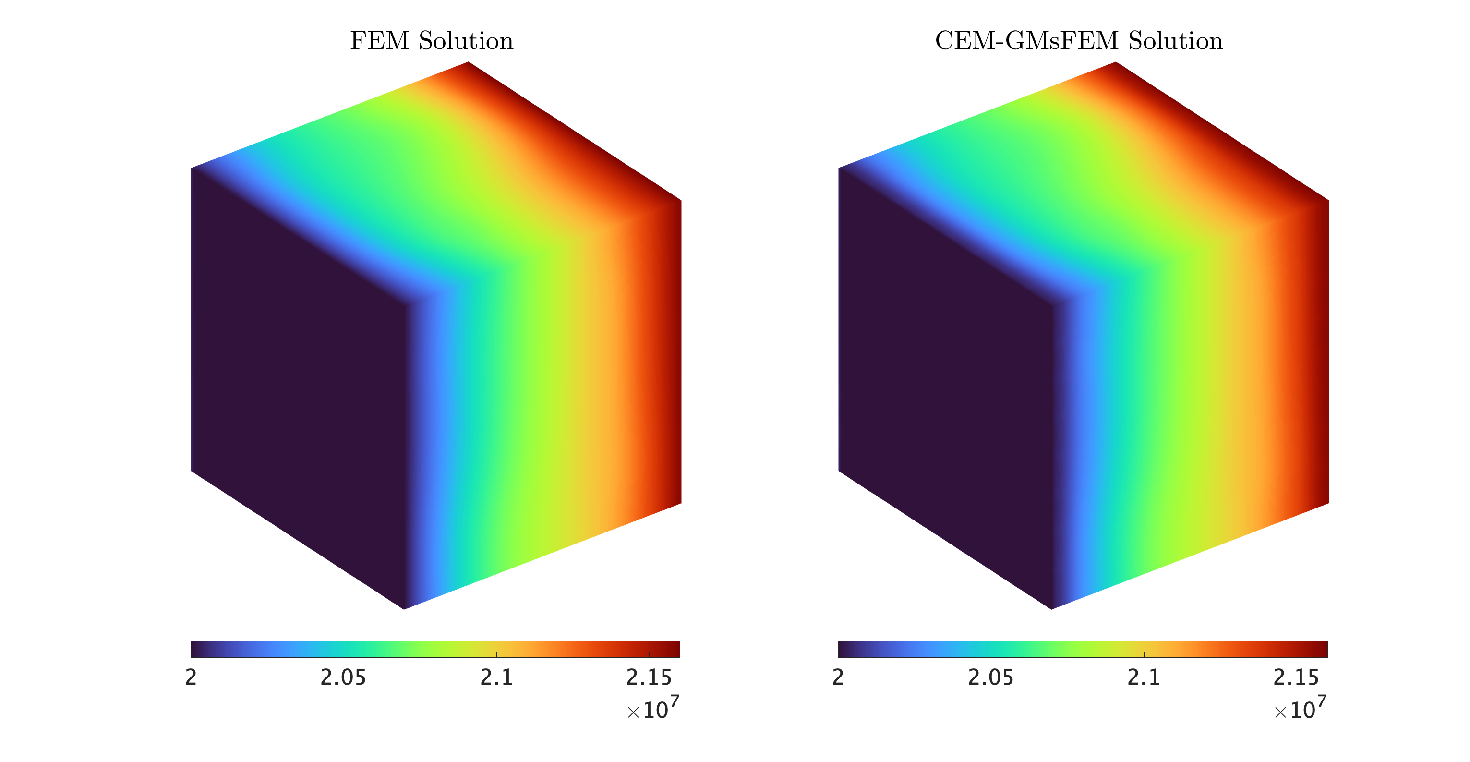}}\\
\subfloat[example2][$t=140$.]{\includegraphics[width=0.8\textwidth]{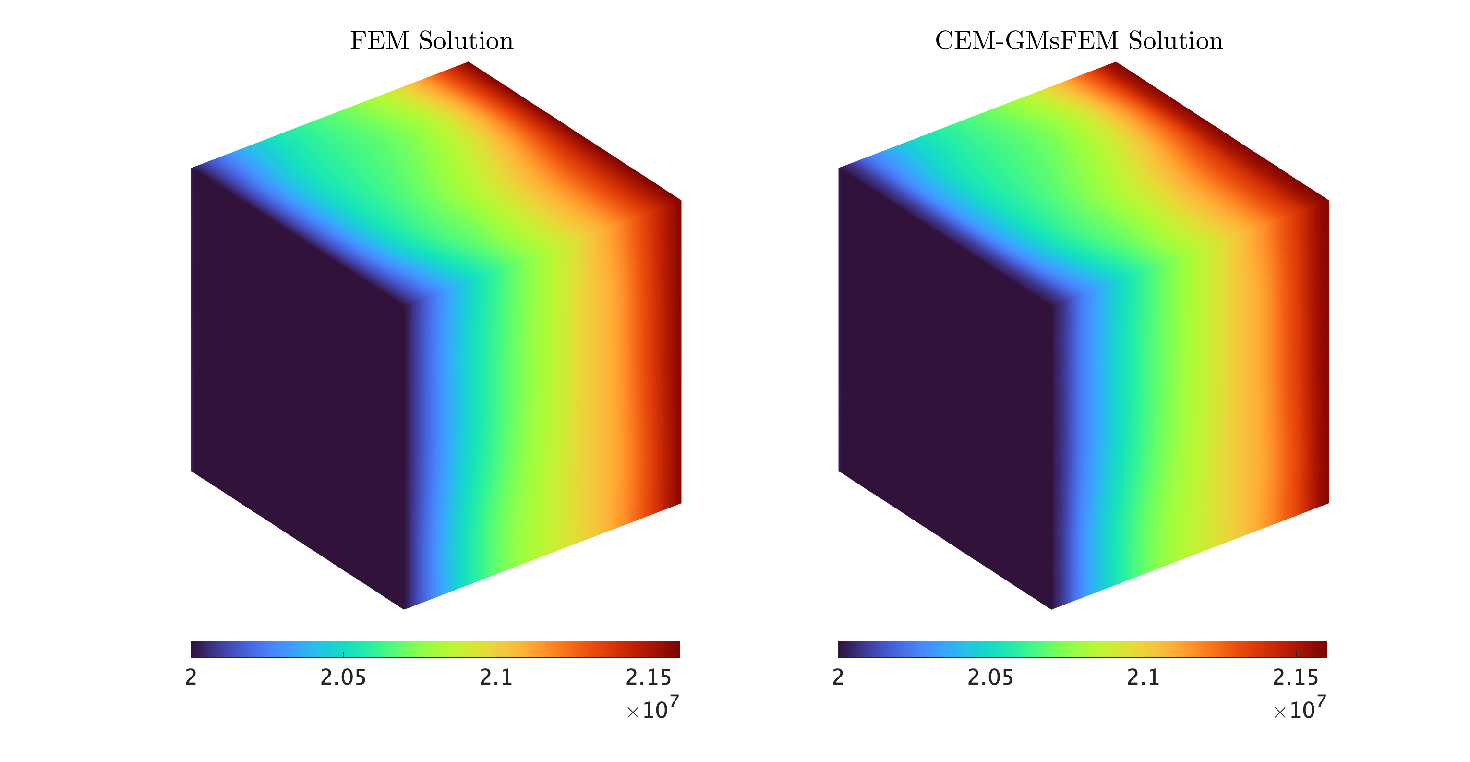}}

\caption{\label{fig:example2} Numerical solution of Example~\ref{example:2} combining a zero Neumann boundary condition and nonzero Dirichlet boundary condition. High-contrast permeability field $\k_{2}$, fine-scale reference solution (left), and CEM-GMsFEM solution (right) with $4$ basis function and $4$ number of oversampling layers at (a) $t=105$ and (b) $t=140$.}
\end{figure}
\end{example}

\begin{example}
\label{example:3}
For the third experiment, we consider the combination of zero Neumann and nonzero Dirichlet boundary conditions as in Example~\ref{example:2}. We set a fine grid resolution of $32^{3}$ (fine-scale system with dimension $35937$), with a size of $h=20$m, and different coarse grid resolutions of $4^{3},8^{3}$ and $16^{3}$ (coarse-scale system with dimension $500$, $2916$ and $19652$ respectively). The time step $\dtn$ and total time simulation are the same as Example~\ref{example:1}. The fractured medium $\k_{3}$ used is depicted in Figure~\ref{fig:perm3}. For this experiment, we employ the framework from \cite{efendiev2015hierarchical} and apply the CEM-GMsFEM to the $3$D model. The domain $\Om$ can be represented by 
\[
\Om = \Om_{0}\cup(\cup_{i}\Om_{\fract,i}),
\]
where $\Om_{0}$ represents the matrix and subscript $\fract$ denotes the fracture regions. Then, we can write the finite element discretization corresponding to equation \eqref{eq:weak-prob}
\begin{align*}
(\phi\pt\rho(\ph),v)+\left(\tfkm\rho(\ph)\gd\ph,\gd v\right)&=(\phi\pt\rho(\ph),v)_{\Om_{0}}+\sum_{i}(\phi\pt\rho(\ph),v)_{\Om_{\fract,i}}\\
&\quad+\left(\tfkm\rho(\ph)\gd\ph,\gd v\right)_{\Om_{0}}+\sum_{i}\left(\tfkm\rho(\ph)\gd\ph,\gd v\right)_{\Om_{\fract,i}}\\
&=(q,v),\quad\feac v\in\Vh,
\end{align*}
In Table~\ref{tab:example3}, we give the convergence rate $T=25\dtn(=175 \text{ days})$ with different coarse-grid sizes $H$. We notice that the error significantly decreases as the size of the coarse grid is finer. Then, we have that the CEM-GMsFEM gives a good approximation of the solution for the case of the fractured medium. Figure~\ref{fig:example3} shows the numerical solutions at $t=70$ and $t=140$. 
\begin{table}[!h]
\centering
\begin{tabular}{|c|c|c|c|c|}
\hline
Number of basis & $H$ &  Number of oversampling layers $m$ & $\varepsilon_{0}$ & $\varepsilon_{1}$\\
\hline
 4 & $4h$ & $3$ &1.6111E-03 & 2.9441E-01\\
 4 & $8h$ & $4$ & 1.9532E-04 & 1.1653E-01\\
 4 & $16h$ & $5$ & 1.0160E-04 & 5.1492E-02\\
\hline
\end{tabular}
\caption{\label{tab:example3} Convergence rate of Example~\ref{example:3} with different numbers of oversampling layers ($m$) with a combination of zero Neumann and nonzero Dirichlet boundary conditions.}\label{tab:example3}
\end{table}

\begin{figure}[!h]
\centering
\subfloat[example3][$t=70$.]{\includegraphics[width=0.8\textwidth]{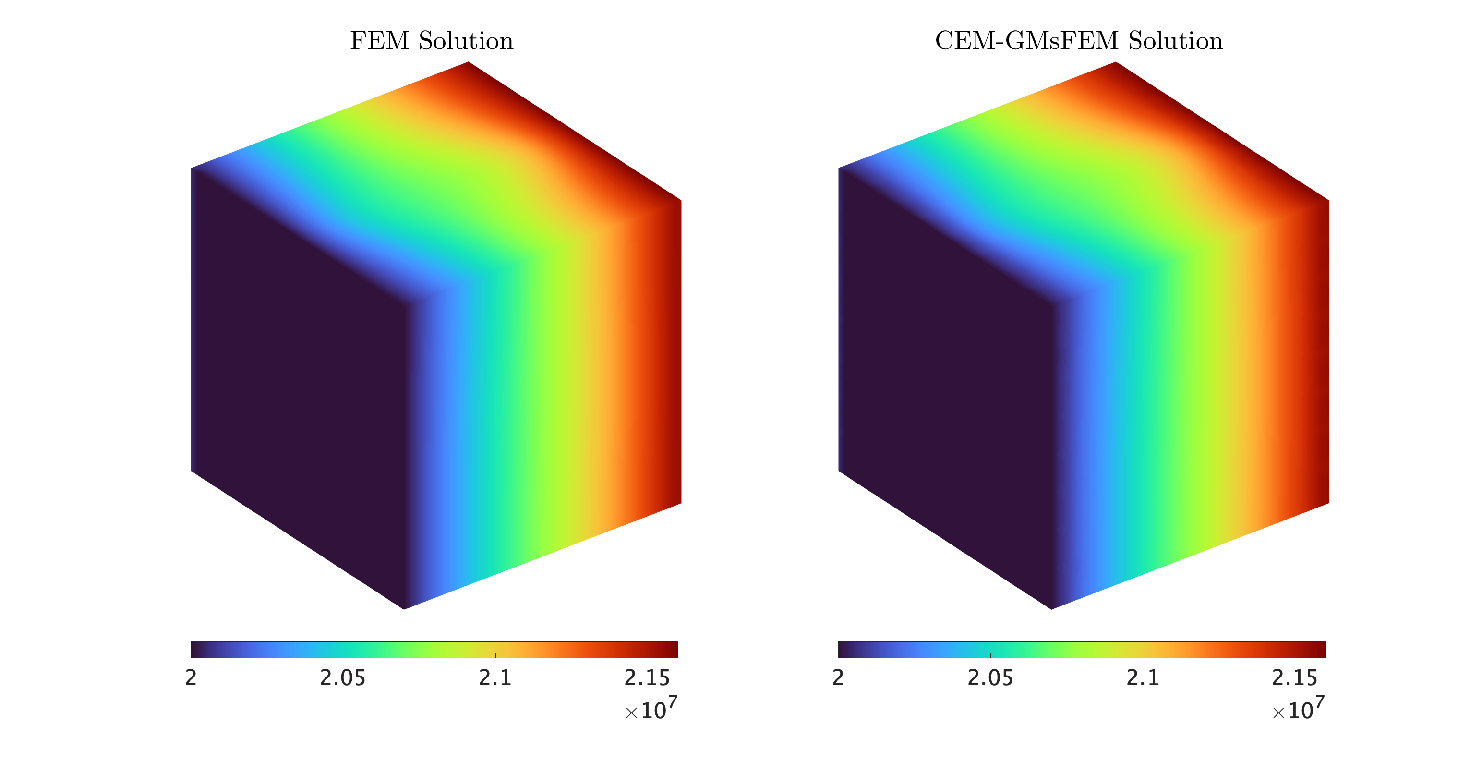}}\\
\subfloat[example2][$t=140$.]{\includegraphics[width=0.8\textwidth]{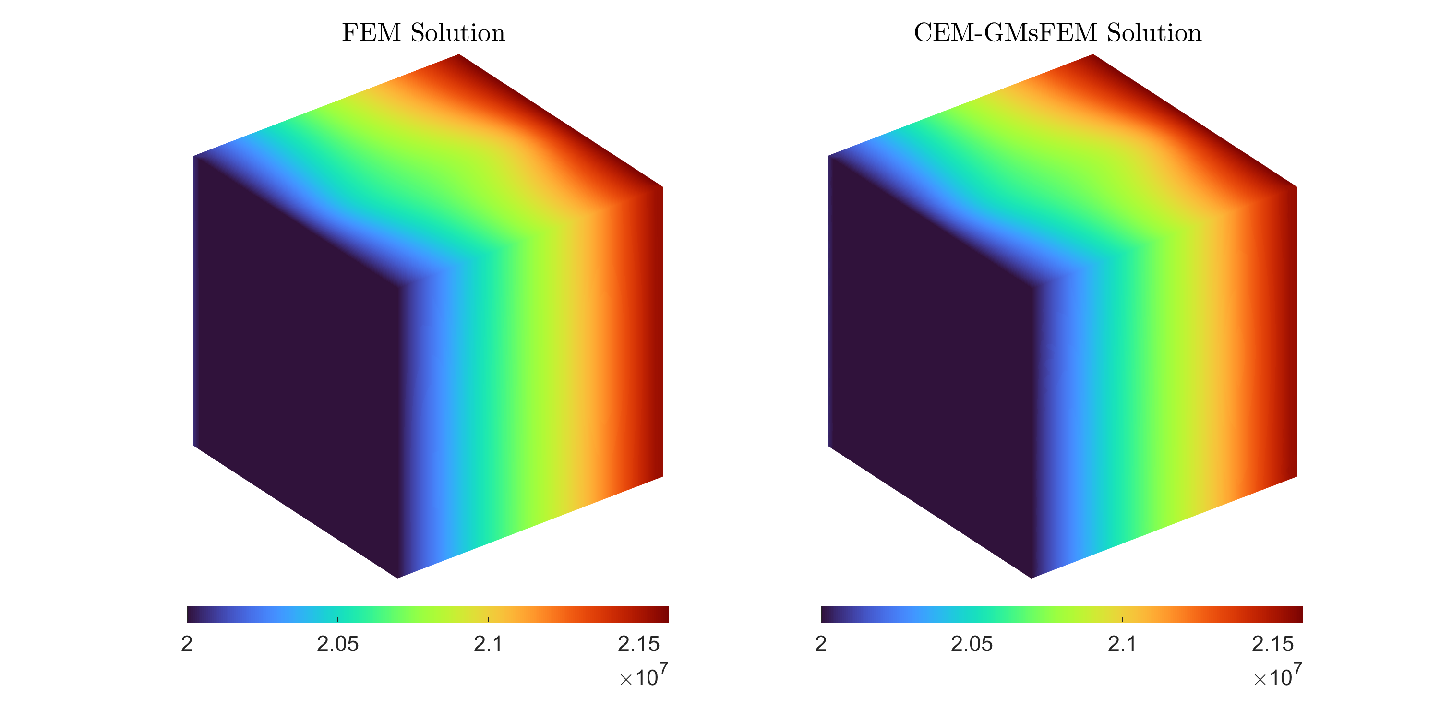}}

\caption{\label{fig:example3} Numerical solution of Example~\ref{example:3} combining a zero Neumann boundary condition and nonzero Dirichlet boundary condition. High-contrast permeability field $\k_{3}$. The fine-scale reference solution (left) and CEM-GMsFEM solution (right) with $4$ basis function and $4$ number of oversampling layers at (a) $t=70$ and (b) $t=140$.}
\end{figure}
\end{example}

\begin{example}
\label{example:4}
For the last experiment, we consider the same boundary conditions as Example~\ref{example:2}. The permeability field used is $\k_{4}$ (see Figure~\ref{fig:perm4}) and the time step $\dtn$ is $7$ days, with total time simulation will be $T=26\dtn(=186\mbox{ days})$. We set a fine grid resolution of $220\times60\times30$ (fine-scale system with a dimension of $417911$), with a size of $h=20$m, and coarse grid resolution of $10^{3}$ (coarse-scale system with a dimension $5324$). We show the pressure profiles comparison in Figure~\ref{fig:example4}. This experiment obtains an error estimate of $\varepsilon_{0}=$1.8377E-03 and $\varepsilon_{1}=$3.4547E-01, respectively.
\begin{figure}[!h]
\centering
\subfloat[example4][]{\includegraphics[width=0.8\textwidth]{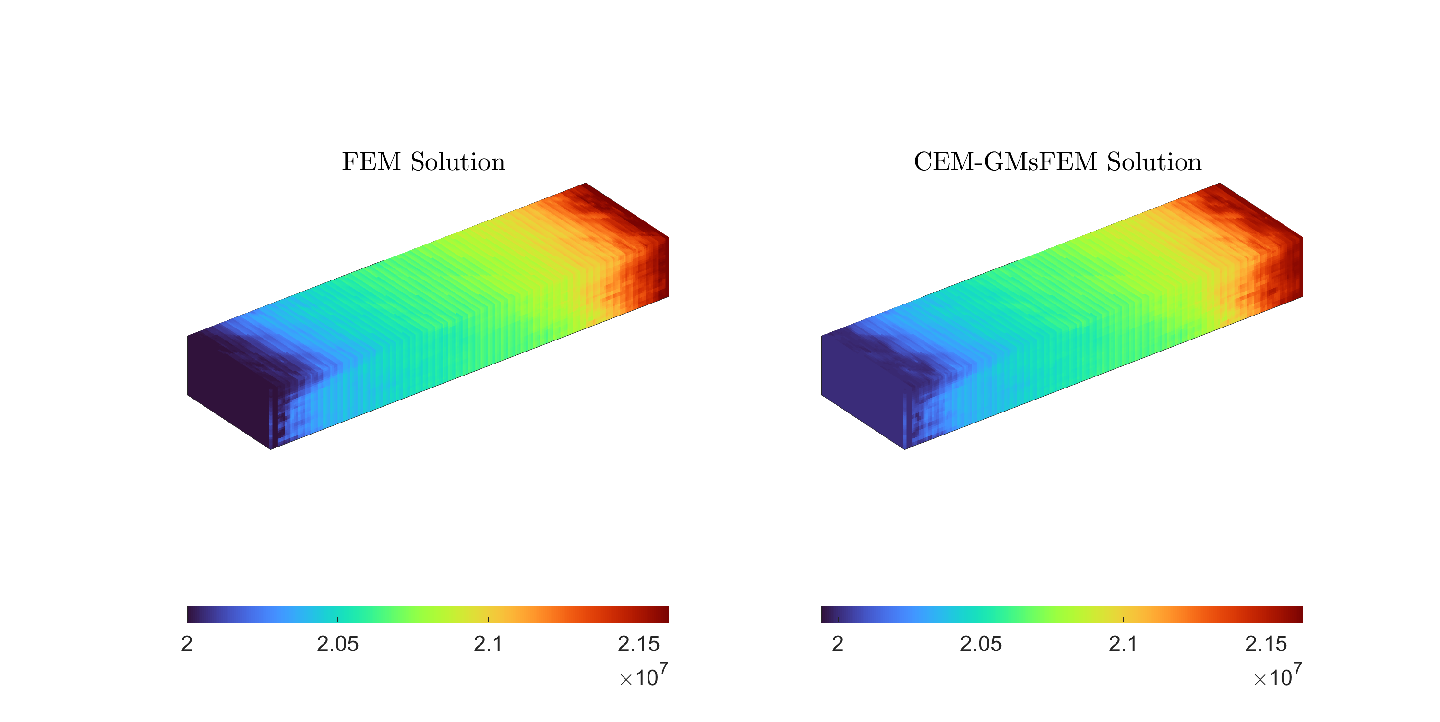}}\\
\subfloat[example4][]{\includegraphics[width=0.8\textwidth]{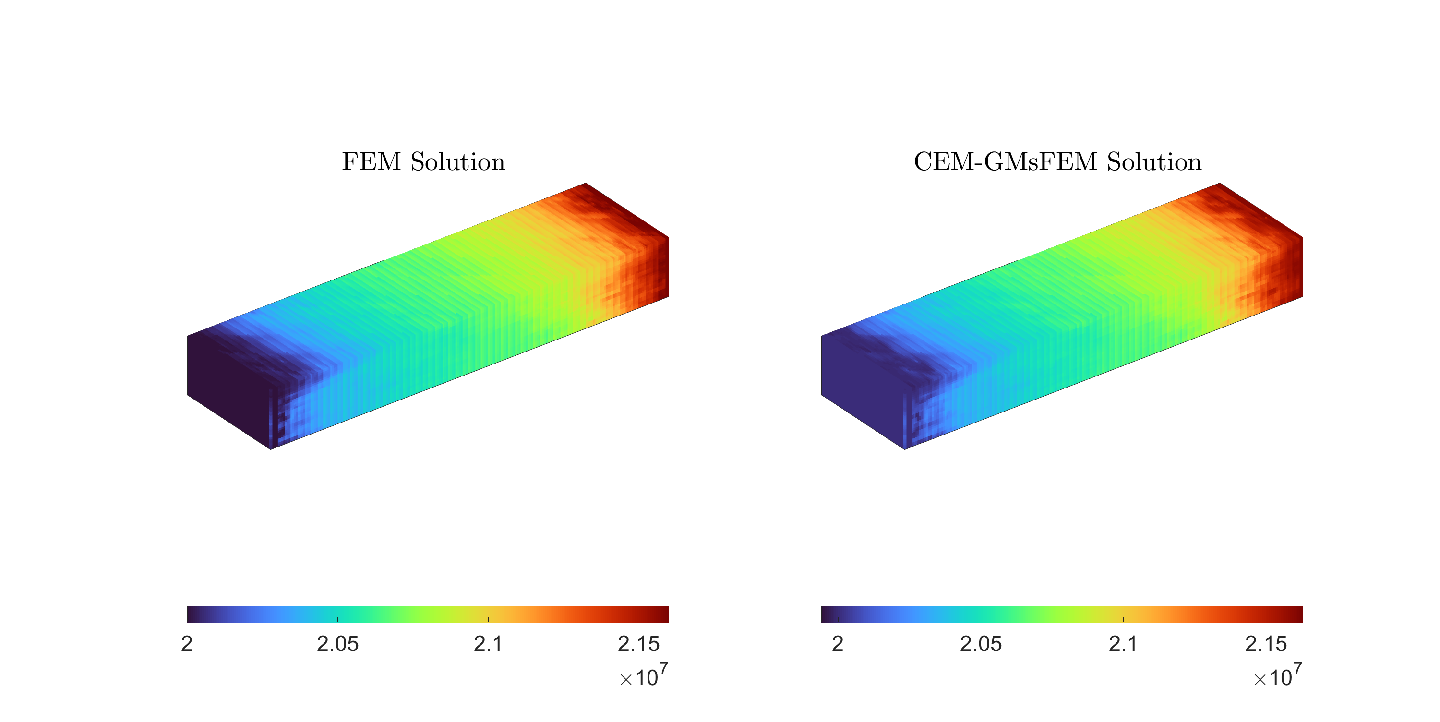}}
\caption{\label{fig:example4} Example 4. Mixed boundary conditions (a) $t=84$ and (b) $t=126$.}
\end{figure}
\end{example}
\section{Concluding remarks}
\label{sec:conclusion}
This paper studies the convergence of the numerical approximations to the highly heterogeneous nonlinear single-phase compressible flow by CEM-GMsFEM. We first build an auxiliary for the proposed method by solving spectral problems. Then, we construct a multiscale basis function by solving some constraint energy minimization problems in the oversampling local regions. So, we obtain multiscale basis functions for the pressure. This work defines the elliptic projection in the multiscale space spanned by CEM-GMsFEM basis functions for convergence analysis. Thus, we present the convergence of the semi-discrete formulation. The convergence depends on the coarse mesh size and the decay of eigenvalues of local spectral problems; an \emph{a posteriori} error estimate is derived underlining discretization. Some numerical examples have been presented to verify the feasibility of the proposed method concerning convergence and stability. We observe that the CEM-GMsFEM is shown to have a second-order convergence rate in the $L^{2}$-norm and a first-order convergence rate in the energy norm concerning the coarse grid size.

A foreseeable result in ongoing research is to boost the performance of the coarse-grid simulation, mainly where the source term is singular; one may need to further improve the accuracy of the approximation without additional refinement in the grid. We can enrich the multiscale space for such a goal by adding additional basis functions in the online stage \cite{chung2018fast}. These new multiscale basis functions are constructed on the oversampling technique and the information on local residuals. Consequently, we could present an adaptive enrichment algorithm to reduce error in some regions with large residuals.

\section*{Acknowledgement}

The research of Eric Chung is partially supported by the Hong Kong RGC General Research Fund (Projects: 14305222 and 14304021). 

\bibliographystyle{plain}
\bibliography{CEM-GMsFEM_compress_flow}
\end{document}